\documentclass[12pt]{article}%
\usepackage{graphicx}
\usepackage{amsmath}
\usepackage{amsfonts}
\usepackage{amssymb}%
\providecommand{\U}[1]{\protect\rule{.1in}{.1in}}
\newtheorem{theorem}{Theorem}

\newtheorem{corollary}[theorem]{Corollary}

\newtheorem{lemma}[theorem]{Lemma}

\newtheorem{proposition}[theorem]{Proposition}

\newenvironment{proof}[1][Proof]{\noindent\textbf{#1.} }{\ \rule{0.5em}{0.5em}}
\begin{document}

\title{\textbf{Propagation of Chaos }\\\textbf{and }\\\textbf{Poisson Hypothesis\footnotetext{Part of this work has been carried out
in the framework of the Labex Archimede (ANR-11-LABX-0033) and of the A*MIDEX
project (ANR-11-IDEX-0001-02), funded by the \textquotedblleft
Investissements d'Avenir" French Government programme managed by the French
National Research Agency (ANR). Part of this work has been carried out at
IITP RAS. The results of Sections 3--6 were obtained with the support of
Russian Foundation for Sciences (project No. 14-50-00150), which is
gratefully acknowledged.}}}
\author{Serge Pirogov$^{\dag\dag},$ Alexander Rybko$^{\dag\dag}$,
\and Senya Shlosman$^{\dag,\dag\dag}$ and Alexander Vladimirov$^{\dag\dag}$\\$^{\dag}$Aix Marseille Univ, Universit\'{e} de Toulon, \\CNRS, CPT, Marseille, France;\\$^{\dag\dag}$Inst. of the Information Transmission Problems,\\RAS, Moscow, Russia }
\maketitle

\begin{abstract}
We establish the Strong Poisson Hypothesis for symmetric closed networks.
In particular, the asymptotic independence of the nodes as the size of the
system tends to infinity is proved.
\end{abstract}

\section{Introduction}

In this paper we consider simple symmetric closed networks consisting of $N$
servers and $M$ customers. Each server has its own infinite buffer, where the
customers are queuing for service -- so there are $N$ queues. The service
discipline in all queues is FIFO with i.i.d. service times with the
distribution function $F(x)$, $0\leq x<\infty$. We  list the conditions on $F$ in Section \ref{SS1}.
They include the continuity of its density $f(x)$ and the finiteness of the second moment.

The network is maximally symmetric; each customer that has finished its
service at some server is placed afterwords at the end of one of the $N$
queues with probability $1/N$. It can be described by a Markov process.
Namely, for each queue $i=1,...,N$ let us consider the elapsed time $t_{i}$
of service of the customer which is on service now. The state of Markov
process ${\mathcal{A}}_{N,M}(t)$ is the set of lengths of $N$ queues (which
are integers) and the set of elapsed times of service of all customers which
are on service now.

The general algebraic structure governing the symmetric network is the
symmetry group of the Markov process. In our case it is the permutation group
of $N$ elements. If some group $G$ acts on the phase space $X$ of the Markov
process and the transition probabilities are $G$-invariant then we can pass
to the factor-process, i.e. the Markov process on the space $X/G$ of the
orbits of the group $G$.

In our case the state of the Markov process is the sequence $(x_{1}%
,x_{2},...,x_{N})$ where $x_{i}$ is a pair, consisting of the length of the
$i$-th queue and the elapsed service time $t_{i}$ for $i$-th server. The
orbit of this state can be interpreted as an atomic probability measure that
assigns the mass $1/N$ to each $x_{i}$ (empirical measure). So our Markov
process can be factorized to a process on empirical measures. The general
fact proved for a broad class of symmetric queueing systems is the
convergence of the process on empirical measures to the deterministic
evolution of measures $m(t)$ as $N\rightarrow\infty$ (and
$M\rightarrow\infty$ in our case), see \cite{KR}, \cite{RS05}, \cite{BRS}.

For given $N$ and $M$ the service process is ergodic. Indeed, let us register
the following event: all the customers are collected in the same queue and
the first one just starts its service. This renewal event will happen with
probability $1$ and with the finite mean waiting time. So the ergodicity
follows.

The equilibrium distribution $Q_{N,M}$ of this Markov process is symmetric,
i.e. the states of queues in equilibrium are exchangeable random variables.
Our problem is the study of asymptotic properties of this equilibrium
distribution $Q_{N,M}$ as $M$ and $N$ tend to infinity. (In the general case
there is no explicit formula for the joint distribution of queues as well as
for its marginal distributions.)

We want to find the conditions under which the limit
$\lim_{N\rightarrow\infty}Q_{N,aN}$ exists and has the `Propagation of Chaos'
(PoC) property. PoC property means that under $Q_{N,M}$ different nodes of
our network are asymptotically independent. We prove PoC in Section
\ref{PoC}. A similar but different case of PoC was addressed in
\cite{BLP1,BLP2}.

The PoC property for the stationary measure is a part of Strong Poisson
Hypothesis (SPH) formulated below. To formulate this hypothesis let us
consider a single server. Let the inflow to this server be a stationary
Poisson flow with intensity $\lambda<1$. The outflow is a stationary
(non-Poisson) flow with the same intensity $\lambda$.

The distribution of the state of the server is the stationary measure of
$M/G/1/\infty$ system with the input intensity $\lambda$.

The parameter $\lambda$ is found from the following argument: the expectation
of the length of the queue in the stationary $M/G/1/\infty$ system with the
input intensity $\lambda$ is equal to $a$.

SPH claims that in the stationary state in the limit as $N\rightarrow\infty$,
$M/N\rightarrow a$, the empirical measure (of server states
$x_{1},\dots,x_{N}$) tends weakly in probability to the distribution of the
state of the server described above. Moreover, all the servers in the limit
as $N\rightarrow\infty$ become asymptotically independent. For a non-random
service time this hypothesis was proved in a seminal paper by A. Stolyar
\cite{St1}.

Now we will remind the reader about the Non-Linear Markov Processes -- NLMP
-- which describe the limiting properties of the processes ${\mathcal{A}}%
_{N,M}(t)$ as $M,N\rightarrow\infty.$ This NLMP (in the sense of \cite{MK})
denoted below by ${\mathfrak A}(t)$ has the following structure. There is a
single server $M(t)/GI/1/\infty$ with a non-stationary Poisson input flow of
rate $\lambda(t)$. This rate equals the expected value of the output flow of
the server in the state corresponding to the measure $m(t)$, which is the
distribution of the state of the queue at time $t$. This equality defines the
measure $m(t)$ in a unique way if the initial measure $m(0)$ is given.

The theorem about the existence and uniqueness of NLMP was proved in
\cite{KR}. It was generalized to a broad class of symmetric queueing systems
in \cite{RS05}, \cite{BRS}. The convergence of empirical measures for any
finite time $t$ can be interpreted as a functional law of large numbers: on
any finite time interval $[0,T]$ the random evolution of empirical measures
converges in probability to the deterministic evolution $m(t)$ if the initial
empirical measure converges to $m(0)$.

From the convergence of empirical measures and from the general properties of
exchangeable random variables the Weak Poisson Hypothesis follows:

\textbf{WPH} \textit{Suppose that the initial distributions of queues is
symmetric and initial empirical measures converge in probability to the
measure }$m(0)$\textit{ as }$N\rightarrow\infty$\textit{ and }$M\rightarrow
\infty$\textit{. Then at any time }$t$\textit{ the servers are asymptotically
independent (i.e. PoC holds) and the distribution of the queue at any server
is close to the measure }$m(t)$\textit{ determined by the non-linear Markov
process.}

In comparison, the SPH in terms of empirical measures claims, that in the
stationary regime the empirical measures satisfies the law of large numbers,
i.e. converge to some non-random measure as $N\rightarrow\infty,$
$M\rightarrow\infty,$ $M/N\rightarrow a.$ In other words, the limiting
invariant measure of the Markov processes on empirical measures -- which is a
measure on measures -- is in fact concentrated at one single measure, as
$N\rightarrow\infty$, $M\rightarrow\infty$, $M/N\rightarrow a$.

From the general argument (known as Khasminski lemma, see, e.g. [L]) it
follows that if the sequence of Markov processes converges to some
deterministic evolution, then any limit point of the sequence of invariant
measures of these processes is an invariant measure of the limit dynamical
system. So in order to prove SPH it is good to know the invariant measures of
NLMP. This problem was considered in \cite{RS05}: for any given value of the
parameter $a$ -- the mean queue length -- the limiting dynamical system
(NLMP) $\mathfrak{A}(t)$ has a unique fixed point.

Summarizing, our strategy to prove the SPH for the measures $Q_{N,M}$ is the
following.

\begin{enumerate}
\item We check first the convergence of the processes
    ${\mathcal{A}}_{N,M}(t)$ to the NLMP $\mathfrak{A}(t)$ as
    $N\rightarrow\infty$ (this is the statement of WPH).

\item We check further that the dynamical system $\mathfrak{A}(t)$ has a
    unique stationary point, $\nu_{a}$, which is a global attractor on any
    ``leaf" where the mean queue length equals $a$.

\item We check, finally, that the limit points of the (precompact) family
    $Q_{N,M}$ as $N\rightarrow\infty$ and $\frac{M}{N}\rightarrow a,$
    which, in principle, could be mixtures of $\nu_{\tilde{a}}$-s with
    $\tilde{a}\leq a$, are in fact just the measure $\nu_{a}$ itself.
\end{enumerate}

The difficult part of the program is to show that the invariant measures of
the dynamical system (NLMP) are just fixed points. This was established in
\cite{RS05} in some cases, by using the flow smoothing property of
$M(t)/GI/1/\infty$ queueing system. This fact is not true in general: for some
symmetric queueing systems (with several types of servers and several types of
customers) there exist non-atomic invariant measures of the NLMP supported by
non-trivial attractors. In \cite{RSV} the corresponding example with $3$ types
of servers and $3$ types of customers was presented. In case of the simple
symmetric closed network considered here, an example of non-trivial attractor
(but with unbounded function $\beta(x)$) was constructed in \cite{RS08}.

At the end of the article we present a general scheme connecting the
asymptotic independence of exchangeable random variables with the law of
large numbers for empirical measures (a de Finetti-type theorem, see also
\cite{PP}).

\section{Single node}

\subsection{State space}\label{SS1}

The basic element of our model is a server with a queue. The customers arrive
to the queue and are served in the order of arrivals (FIFO service
discipline). The random service time $\eta$ of each customer is i.i.d. with
the distribution $F\left(  x\right)  $. The service discipline is FIFO with
i.i.d. service times with the distribution function $F(x)$, $0\leq x<\infty$.
This paper relies on our previous results concerning the PH, \cite{RS05},
which require some restrictions on $\eta.$ We list now the properties of
$\eta$ needed.

\begin{enumerate}
\item the density function $p\left(  t\right)  $ of random variable $\eta$
    is defined on $t\geq0$ and uniformly bounded from above; moreover, it
    is differentiable in $t,$ with $p^{\prime}\left(  t\right)  $
    continuous;

\item $p\left(  t\right)  $ satisfies the following strong Lipschitz
condition: for some $C<\infty$ and for all $t\geq0$
\begin{equation}
\left\vert p\left(  t+\Delta t\right)  -p\left(  t\right)  \right\vert \leq
Cp\left(  t\right)  \left\vert \Delta t\right\vert , \label{02}%
\end{equation}
provided $t+\Delta t>0$ and $\left\vert \Delta t\right\vert <1;$

\item for some $\delta>0$
\begin{equation}
M_{\delta}\equiv\mathbb{E}\left(  \eta\right)  ^{2+\delta}<\infty, \label{009}%
\end{equation}

\item defining the random variables
\[
\eta\Bigm|_{\tau}=\left(  \eta-\tau\Bigm|\eta>\tau\right)  ,\tau\geq0,
\]
and introducing the functions $p_{\tau}\left(  t\right)  $ as the densities of
the random variables $\eta\Bigm|_{\tau},$ we require that the function
$p_{\tau}\left(  0\right)  $ is bounded uniformly in $\tau\geq0,$
\begin{equation}
p_{\tau}\left(  0\right)  \leq\beta<\infty, \label{158}%
\end{equation}

\item the function $\frac{d}{d\tau}p_{\tau}\left(  0\right)  $ is continuous
and bounded uniformly in $\tau\geq0;$

\item the limits $\lim_{\tau\rightarrow\infty}p_{\tau}\left(  0\right)  ,$
$\lim_{\tau\rightarrow\infty}\frac{d}{d\tau}p_{\tau}\left(  0\right)  $ exist
and are finite.\medskip

\item Without loss of generality we suppose that
\begin{equation}
\mathbb{E}\left(  \eta\right)  =1. \label{153}%
\end{equation}

\end{enumerate}

In particular, power law decaying $\eta$-s are allowed.

As a state space ${\mathcal{Q}}$ of queues at a single server we take the set
of pairs $\left(  z,k\right)  ,$ where $z\geq0$ is the elapsed service time of
the customer under the service, $z\in\mathbb{R}^{1},$ and $k$ in an integer.
Of course, the empty state $\varnothing$ is also included in $\mathcal{Q}$ $.$
For the future use we introduce the subspace $\mathcal{Q}_{0}$ of
$\mathcal{Q}$ by%
\begin{equation}
q\in\mathcal{Q}_{0}\text{ iff }q=\left(  0,k\right)  \text{ for some }k\geq0.
\label{36}%
\end{equation}
In words, $\mathcal{Q}_{0}$ consists of queues, where the service of the first
customer is about ot start.

\subsection{Dynamics}

The dynamics is defined by the following simple relations. Suppose we are in a
state
\[
q(t)=(z\left(  t\right)  ,k\left(  t\right)  )\in\mathcal{Q}.
\]
While the time goes and nothing happens, $k$ stays constant, and $z$ grows
linearly: $\dot{z}(t)=1.$ If a customer arrives at the moment $t$, then we
have a jump:%
\[
(z\left(  t\right)  ,k\left(  t\right)  )\rightarrow(z\left(  t\right)
,k\left(  t\right)  +1).
\]
If a customer leaves at the moment $t$, then we have another jump:
\[
(z\left(  t\right)  ,k\left(  t\right)  )\rightarrow(0,k\left(  t\right)
-1).
\]

\section{Mean-field network}

\subsection{Definition}

The mean-field network consits of $N$ nodes, described above. Their collective
behavior is defined as follows. As soon as a customer finishes its service at
some node, it is routed to one of $N$ nodes with equal probability $1/N$ (this
is why the network is called a \emph{mean-field} network). At arrival to this
node the customer joins the queue and waits for its turn to be served. Thus,
the total number of customers,%
\[
M=\sum_{i=1}^{N}k_{i},
\]
is conserved by the dynamics. The resulting Markov process is denoted by
${\mathcal{A}}_{N,M}\left(  t\right)  .$

\subsection{Ergodicity}

For each pair $(N,M)$, we denote by $Q_{N,M}$ the unique equilibrium state of
the process ${\mathcal{A}}_{N,M}\left(  t\right)  $. This is a probability
measure on $\mathcal{Q}^{N}$.

A point $\left(  q_{1},...,q_{N}\right)  $ of the space $\mathcal{Q}^{N}$ can
be conviniently identified with a probability measure $\mu=\frac{1}{N}%
\sum_{n=1}^{N}\delta_{q_{i}}$ on $\mathcal{Q},$ i.e. with an element of
$\mathcal{M}\left(  \mathcal{Q}\right)  .$ In fact, it is an element of the
subspace $\mathcal{M}_{N}\left(  \mathcal{Q}\right)  \subset\mathcal{M}\left(
\mathcal{Q}\right)  $ of the atomic measures with atom weights equal to
$\frac{1}{N}.$ Hence the states of all the processes ${\mathcal{A}}%
_{N,M}\left(  t\right)  ,$ as well as the measures $Q_{N,M},$ are elements of
$\mathcal{M}\left(  \mathcal{M}_{N}\left(  \mathcal{Q}\right)  \right)  .$

\section{Some facts about NLMPs}

\subsection{Non-linear Markov processes}

The NLMPs $\mathfrak{A}_{a}$ are dynamical systems on $\mathcal{M}\left(
\mathcal{Q}\right)  .$ Under $\mathfrak{A}_{a},$ the measures evolve,
informally speaking, in the same way as under ${\mathcal{A}}_{N,aN}$ with $N$
very large. For the details of the NLMP see \cite{KR, RS05, BRS}.

Here is a description of the corresponding dynamical systems. Each of them
acts on the space of states $\mu\in\mathcal{M}\left(  \mathcal{Q}\right)  $ of
a single server. Every initial state $\mu\left(  0\right)  $ defines a certain
function $\lambda_{\mu\left(  0\right)  }(t)\geq0,$ which is the rate of the
arrival of the Poisson process of customers (to a single node). Once given,
the rate $\lambda_{\mu\left(  0\right)  }(t)-$Poissonian inflow defines the
evolution $\mu\left(  t\right)  $ of the state $\mu\left(  0\right)  .$

The definition of $\lambda(t)$ is somewhat complicated. Suppose we know
$\lambda(t).$ Then we know the exit flow (non-Poisson, in general) from our
server. This flow has some rate, $b(t),$ which is defined by $\mu\left(
t\right)  $ as follows. Given $\mu(t)$, let $\nu_{\mu(t)}$ be the probability
distribution of the remaining service time of$\;$ the customer currently
served. Then $b(t)=\lim_{\Delta\rightarrow0}\frac{\nu_{\mu(t)}([0,\Delta
])}{\Delta}.$

The function $\lambda_{\mu\left(  0\right)  }(t)$ is the solution of the
following (non-linear integral) equation:%
\[
\lambda_{\mu\left(  0\right)  }(t)=b(t).
\]
The subscript $a$ in the notation $\mathfrak{A}_{a}$ refers to the fact that
our dynamics conserves the average number $\mathcal{N}\left(  \mu\right)  $ of
customers:
\begin{equation}
\mathcal{N}\left(  \mu\right)  =\int_{\mathcal{Q}}k\ \mu d\left(  q\right)
=a;\ \ q=\left(  z,k\right)  \in\mathcal{Q}. \label{33}%
\end{equation}

In the following we will need one property of our NLMProcesses, which is
proven in Lemma 7 of \cite{RS05}.

\begin{lemma}
\label{Delta} Let $\Delta>0$ and the parameter $a$ of $\left(  \ref{33}%
\right)  $ is fixed. There exists a function $T\left(  \Delta,a\right)  ,$
such that for any $T>T\left(  \Delta,a\right)  $ and for any initial state
$\mu\left(  0\right)  $ satisfying $\left(  \ref{33}\right)  ,$ the
corresponding rate function $\lambda_{\mu\left(  0\right)  }$ satisfies the
estimate%
\begin{equation}
\int_{\tau}^{\tau+T}\lambda_{\mu\left(  0\right)  }\left(  t\right)
\ dt<T-\Delta\label{40}%
\end{equation}
for any $\tau>0.$
\end{lemma}

\subsection{Convergence to NLMP as $N\rightarrow\infty$}

Here we formulate the theorem about finite time convergence of processes
${\mathcal{A}}_{N,M}\left(  t\right)  $ to $\mathfrak{A}_{a}\left(  t\right)
,$ with $M/N\rightarrow a.$ More precisely, for any finite time interval
$[0,T]$, the evolution under ${\mathcal{A}}_{N,M}$ with the initial state
$\mu_{N}(0)$ converge to the evolution under $\mathfrak{A}_{a}\mathfrak{\ }%
$with the same initial state $\mu_{N}(0)$. The convergence here is the weak
convergence, i.e. the convergence of continuous functionals $\mathfrak{f}$ of
the trajectories $\left\{  \mu\left(  t\right)  ,t\in\left[  0,T\right]
\right\}  $.

\begin{theorem}
\label{T7} Let $T>0$ and let $\mathfrak{f}$ be a continuous functional on the
set of the trajectories $\left\{  \mu\left(  t\right)  ,t\in\left[
0,T\right]  \right\}  .$ Suppose that for any $N$ the number $M_{N}$ of
customers is chosen, in such a way that the the sequence $a_{N}=\frac{M_{N}%
}{N}$ has a limit. Then for any family of initial states $\left\{  \mu
_{N}(0)\in\mathcal{M}_{N}\left(  \mathcal{Q}\right)  ,N=1,2,...\right\}  $
with $M_{N}$ customers we have
\begin{equation}
\left\vert \mathfrak{f}\left(  {\mathcal{A}}_{N,M}\mu_{N}\right)
-\mathfrak{f}\left(  \mathfrak{A}_{a_{N}}\mu_{N}\right)  \right\vert
\rightarrow0\text{ as }N\rightarrow\infty, \label{34}%
\end{equation}
uniformly in $\left\{  \mu_{N}(0)\in\mathcal{M}_{N}\left(  \mathcal{Q}\right)
,N=1,2,...\right\}  .$
\end{theorem}

\begin{proof}
See \cite{KR}. (Of course, the convergence in $\left(  \ref{34}\right)  $ is
not uniform in $T.$)
\end{proof}

Actually, we will only need the special case of this theorem, applied to
certain functionals $\mathfrak{f}_{n,T}.$ Let us consider our network of $N$
servers, and fix one of them, $s$. Consider the random variable $\mathcal{C}%
_{N,T,\mu},$ defined as the number of customers coming to the server $s$
during the time interval $\left[  0,T\right]  $ in the process ${\mathcal{A}%
}_{N,M}\mu$ (started from the state $\mu\in\mathcal{M}_{N}\left(
\mathcal{Q}\right)  $). In the same way we define the random variable
$\mathcal{C}_{T,\mu},$ as the number of customers coming to the server $s$
during the time interval $\left[  0,T\right]  $ in the NLMProcess
$\mathfrak{A}_{\frac{M}{N}}\mu.$

\begin{corollary}
\label{K} Under conditions of Theorem \ref{T7} we have: for each $n,T$
\begin{equation}
\left\vert \mathbf{\Pr}\left(  \mathcal{C}_{N,T,\mu}=n\right)  -\mathbf{\Pr
}\left(  \mathcal{C}_{T,\mu}=n\right)  \right\vert \rightarrow0 \label{50}%
\end{equation}
as $N\rightarrow\infty.$
\end{corollary}

\subsection{Convergence of NLMP-s as T$\rightarrow\infty$}

In this section we formulate the convergence properties of our NLMProcesses,
which will be crucially used later. In words, they state that the trajectory
of our dynamical systems $\left(  \mathfrak{A}_{a}\mu\right)  \left(
t\right)  $ go to the limit, as $t\rightarrow\infty$ (and not to some more
complicated limit set). This statement is the content of the main Theorem 1 of
\cite{RS05}.

\begin{proposition}
Suppose that the measure $\mu$ on $\mathcal{Q}$ satisfies

$i)$ $\mathcal{N}\left(  \mu\right)  =a,$

$ii)$ $\mu$ is supported by $\mathcal{Q}_{0}\subset\mathcal{Q}.$

Then the limit $\lim_{t\rightarrow\infty}\left(  \mathfrak{A}_{a}\mu\right)
\left(  t\right)  $ exists and equals to the measure $\nu_{a},$ which is the
unique stationary point of the evolution $\mathfrak{A}_{a}.$
\end{proposition}

Note that the Theorem 1 of \cite{RS05} can be applied only to the initial
measures $\mu$ which satisfy an extra condition (20) of the paper \cite{RS05}.
But for the measures supported by $\mathcal{Q}_{0}$ it holds evidently.

The stationary measures $\nu_{a}$ satisfy $\mathcal{N}\left(  \nu_{a}\right)
=a;$ they are uniquely defined by the random service time distribution $\eta.$

\section{Convergence $Q_{N,M}\rightarrow\nu_{a}$}

Here we will prove the convergence:
\begin{equation}
\text{if }\lim_{N\rightarrow\infty}\frac{M_{N}}{N}=a,\text{ then }%
\lim_{N\rightarrow\infty}Q_{N,M_{N}}=\nu_{a},\label{52}%
\end{equation}
where the equilibrium measures $\nu_{a}$ were introduced in the previous
section. Since the expectations $\mathcal{N}\left(  Q_{N,M_{N}}\right)
\rightarrow a$, the family $Q_{N,M_{N}}$ is compact. Therefore to prove
$\left(  \ref{52}\right)  $ it is enough to show that for every limit point
$\lim_{n\rightarrow\infty}Q_{N_{n},M_{N_{n}}}$ of the family $Q_{N,M_{N}}$ we
have
\begin{equation}
\mathcal{N}\left(  \lim_{n\rightarrow\infty}Q_{N_{n},M_{N_{n}}}\right)
=a\label{53}%
\end{equation}
(and so the stationary measure $\lim_{n\rightarrow\infty}Q_{N_{n},M_{N_{n}}}$
of the process $\mathfrak{A}$ is $\nu_{a}.$) In general one can claim only
that $\mathcal{N}\left(  \lim_{n\rightarrow\infty}Q_{N_{n},M_{N_{n}}}\right)
\leq a,$ but we are going to define a process $\mathfrak{B}_{T}$ which
dominates all the processes ${\mathcal{A}}_{N,M}$ with $N$ large enough, as
well as their stationary states $Q_{N,M}.$ Since its stationary distribution
$\mathfrak{\bar{B}}$ has finite expected number of customers $\mathcal{N}%
\left(  \mathfrak{\bar{B}}\right)  ,$ the family $Q_{N,M}$ is uniformly
integrable, so $\left(  \ref{53}\right)  $ follows.

In order to define the process $\mathfrak{B}_{T}$  we will use the notion of
measure dominance and we introduce some notations.

Let $\xi,\zeta$ be two probability measures on $\mathbb{Z}_{+}^{1}=\left\{
0,1,2,...\right\}  .$

\begin{enumerate}
\item We say that $\zeta$ dominates $\xi,$ i.e. $\xi\preceq\zeta$ iff for any
$n>0$
\begin{equation}
\xi\left(  \left[  0,n\right]  \right)  \geq\zeta\left(  \left[  0,n\right]
\right)  . \label{49}%
\end{equation}
(In words, $\zeta$ is to the right of $\xi.$)

\item For $l>0$ we say that $\xi\preceq_{l}\zeta$ iff $\left(  \ref{49}%
\right)  $ holds for $n\leq l.$
\end{enumerate}

For any $\xi\preceq\zeta$ and every $k\in\mathbb{Z}_{+}^{1}$ we now define
measures $\xi\diamond_{k}\zeta,$ , so that $\zeta=\xi\diamond_{0}\zeta
\succeq\xi\diamond_{1}\zeta\succeq...\succeq\xi.$ The probability measure
$\xi\diamond_{k}\zeta$ is uniquely defined by the properties:

\begin{enumerate}
\item $\xi\diamond_{k}\zeta\preceq\zeta,$

\item for every $n<k$ we have $\left(  \xi\diamond_{k}\zeta\right)  \left(
\left[  0,n\right]  \right)  =\xi\left(  \left[  0,n\right]  \right)  ,$

\item for every $n>K\left(  \xi,\zeta\right)  >k$ we have $\left(  \xi
\diamond_{k}\zeta\right)  \left(  [n,+\infty)\right)  =\zeta\left(  \lbrack
n,+\infty)\right)  ,$ where the integer $K\left(  \xi,\zeta\right)  $ satisfies

\item $\left(  \xi\diamond_{k}\zeta\right)  \left(  \left[  k,K\left(
\xi,\zeta\right)  -1\right]  \right)  =0.$

[\textbf{Comment. }The last relation defines the value $\left(  \xi
\diamond_{k}\zeta\right)  \left(  K\left(  \xi,\zeta\right)  \right)  $ to be
equal to $1-\xi\left(  \left[  0,k-1\right]  \right)  -\zeta\left(  \lbrack
K\left(  \xi,\zeta\right)  +1,+\infty)\right)  .$ This value does not exceed
$\zeta\left(  K\left(  \xi,\zeta\right)  \right)  .\ $]
\end{enumerate}

With these notations we have a simple lemma:

\begin{lemma}
\label{mix} Suppose that $\xi\preceq\zeta,$ and for a random variable
$\varkappa\geq0$ and some $k>0$ we have:%

\[
\varkappa\preceq\zeta\text{ and }\varkappa\preceq_{k}\xi.
\]

Then%
\[
\varkappa\preceq\xi\diamond_{k+1}\zeta
\]

\end{lemma}

$\blacksquare$

Now we will construct a stationary process which dominates all the processes
${\mathcal{A}}_{N,M}$ with $N$ large enough, as well as their stationary
states $Q_{N,M}.$

Let us fix some value $\Delta>0$ (compare with Lemma \ref{Delta}; for example,
$\Delta=1$ would go), and fix some $T>T\left(  \Delta,a\right)  .$ Consider
the discrete random variable $\chi_{T-\frac{\Delta}{2}},$ which has Poisson
distribution with parameter $T-\frac{\Delta}{2}.$ Let us pick a small positive
$\varepsilon>0,$ to be specified later, and define the integer $K$ as the one
satisfying%
\[
\mathbf{\Pr}\left(  \chi_{T-\frac{\Delta}{2}}>K\right)  <\varepsilon.
\]
According to the theorem \ref{T7}, its corollary \ref{K} and lemma \ref{Delta}
we know that for all $N$ large enough and for any initial state $\mu
\in\mathcal{M}_{N}\left(  \mathcal{Q}\right)  $we have%
\[
\mathcal{C}_{N,T,\mu}\preceq_{K}\chi_{T-\frac{\Delta}{2}}.
\]
We have also a straghtforward relation%
\[
\mathcal{C}_{N,T,\mu}\preceq\chi_{T\beta},
\]
see $\left(  \ref{158}\right)  .$ Applying lemma \ref{mix}, we conclude that
\begin{equation}
\mathcal{C}_{N,T,\mu}\preceq\chi_{T-\frac{\Delta}{2}}\diamond_{K+1}%
\chi_{T\beta}, \label{78}%
\end{equation}
(provided $\beta>1$). What is very important for us is that
\begin{equation}
\mathbb{E}\left(  \chi_{T-\frac{\Delta}{2}}\diamond_{K+1}\chi_{T\beta}\right)
\leq\left(  T-\frac{\Delta}{2}\right)  +\varepsilon T\beta<T-\frac{\Delta}{4}
\label{79}%
\end{equation}
once $\varepsilon$ is small enough.

Consider now the random queueing process $\mathfrak{B}_{T}$, when the
customers are arriving in groups only at discrete moments $kT,$ $k=0,1,2,...$,
while the number of customers in groups is iid, with distribution
$\chi_{T-\frac{\Delta}{2}}\diamond_{K+1}\chi_{T\beta}.$ Because of $\left(
\ref{79}\right)  ,$ the process $\mathfrak{B}_{T}$ is ergodic, and because of
$\left(  \ref{78}\right)  $ it dominates all the processes ${\mathcal{A}%
}_{N,M},$ see \cite{BF}. Therefore the stationary distribution $\mathfrak{\bar
{B}}$ of $\mathfrak{B}_{T}$ dominates all the states $Q_{N,M},$ and we are done.

\section{Propagation of Chaos}\label{PoC}

Here we prove finally the Propagation of Chaos property: under $Q_{N,M},$
different nodes of our network are asymptotically independent.

This result follows from the general theorem we will present now. Let $k$ be
fixed, and suppose that for every set of integers $n_{1},...,n_{k}$ a
collection of random variables $\xi_{1}^{1},...,\xi_{n_{1}}^{1},\xi_{1}%
^{2},...,\xi_{n_{2}}^{2},\xi_{1}^{k},...,\xi_{n_{k}}^{k}$ is given. Suppose
that the joint distribution $P_{n_{1},...,n_{k}}$ of this collection is
invariant under the action of the product of the permutation groups $S_{n_{1}%
}\times...\times S_{n_{k}},$ where each group $S_{n_{i}}$ permutes the random
variables $\xi_{1}^{i},...,\xi_{n_{i}}^{i}.$ Suppose that for each $i=1,...,k$
the Law of Large Numbers (LLN) holds for $\xi_{1}^{i},...,\xi_{n_{i}}^{i}$,
which means that for every bounded measurable function $f$ we have that the
average
\[
\frac{1}{n_{i}}\sum_{j=1}^{n_{i}}f\left(  \xi_{j}^{i}\right)  \rightarrow
\mu^{i}\left(  f\right)
\]
in probability, where $\mu^{i}\left(  \ast\right)  $ is some (non-random)
functional. Then the collection $\xi_{1}^{1},...,\xi_{n_{1}}^{1},\xi_{1}%
^{2},...,\xi_{n_{2}}^{2},\xi_{1}^{k},...,\xi_{n_{k}}^{k}$ is asymptotically independendent:

\begin{theorem}
For any $m_{1},...,m_{k}$ and any collection $f_{j}^{i}$ of bounded measurable
functions, $j=1,...,m_{i},\ i=1,...,k,$ the expectation
\begin{equation}
\mathbb{E}_{P_{n_{1},...,n_{k}}}\left(  \prod_{i=1}^{k}\prod_{j=1}^{m_{j}%
}f_{j}^{i}\left(  \xi_{j}^{i}\right)  \right)  \rightarrow\prod_{i=1}^{k}%
\prod_{j=1}^{m_{j}}\mu^{i}\left(  f_{j}^{i}\right)  \label{01}%
\end{equation}
when all $n_{i}\rightarrow\infty.$ Also,%
\[
\prod_{i=1}^{k}\prod_{j=1}^{m_{j}}\mathbb{E}_{P_{n_{1},...,n_{k}}}\left[
f_{j}^{i}\left(  \xi_{j}^{i}\right)  \right]  \rightarrow\prod_{i=1}^{k}%
\prod_{j=1}^{m_{j}}\mu^{i}\left(  f_{j}^{i}\right)  .
\]

\end{theorem}

\begin{proof}
The second claim follows immediately from the LLN for the collections $\xi
_{j}^{i},$ $j=1,...,n_{i},$ since it claims that $\mathbb{E}_{P_{n_{1}%
,...,n_{k}}}\left[  f_{j}^{i}\left(  \xi_{j}^{i}\right)  \right]
\rightarrow\mu^{i}\left(  f_{j}^{i}\right)  $ as $n_{i}\rightarrow\infty.$ To
see $\left(  \ref{01}\right)  ,$ let us save on notations, and consider the
case $k=2,$ $m_{1}=m_{2}=2.$ So we are dealing with the random variables
$\xi_{1},...,\xi_{n},\eta_{1},...,\eta_{m},$ while their joint distribution
$P_{n,m}$ is $S_{n}\times S_{m}$-invariant. Due to the symmetry, the
expectation%
\begin{align*}
&  \mathbb{E}\left[  f_{1}\left(  \xi_{1}\right)  f_{2}\left(  \xi_{2}\right)
g_{1}\left(  \eta_{1}\right)  g_{2}\left(  \eta_{2}\right)  \right] \\
&  =\frac{1}{n\left(  n-1\right)  m\left(  m-1\right)  }\mathbb{E}\left(
\sum_{i\neq j,k\neq l}f_{1}\left(  \xi_{i}\right)  f_{2}\left(  \xi
_{j}\right)  g_{1}\left(  \eta_{k}\right)  g_{2}\left(  \eta_{l}\right)
\right)  .
\end{align*}
Since $f$-s and $g$-s are bounded,
\begin{align*}
&  \frac{1}{n\left(  n-1\right)  m\left(  m-1\right)  }\left\vert
\mathbb{E}\left(  \sum_{i\neq j,k\neq l}f_{1}\left(  \xi_{i}\right)
f_{2}\left(  \xi_{j}\right)  g_{1}\left(  \eta_{k}\right)  g_{2}\left(
\eta_{l}\right)  \right)  -\right. \\
-  &  \left.  \mathbb{E}\left(  \sum_{i,j,k,l}f_{1}\left(  \xi_{i}\right)
f_{2}\left(  \xi_{j}\right)  g_{1}\left(  \eta_{k}\right)  g_{2}\left(
\eta_{l}\right)  \right)  \right\vert \rightarrow0,
\end{align*}
as $m,n\rightarrow\infty.$ But%
\begin{align*}
&  \frac{1}{n\left(  n-1\right)  m\left(  m-1\right)  }\mathbb{E}\left(
\sum_{i,j,k,l}f_{1}\left(  \xi_{i}\right)  f_{2}\left(  \xi_{j}\right)
g_{1}\left(  \eta_{k}\right)  g_{2}\left(  \eta_{l}\right)  \right) \\
&  =\frac{nm}{\left(  n-1\right)  \left(  m-1\right)  }\mathbb{E}\left[
\left(  \frac{1}{n}\sum_{i}f_{1}\left(  \xi_{i}\right)  \right)  \left(
\frac{1}{n}\sum_{j}f_{2}\left(  \xi_{j}\right)  \right)  \left(  \frac{1}%
{m}\sum_{k}g_{1}\left(  \eta_{k}\right)  \right)  \left(  \frac{1}{m}\sum
_{l}g_{2}\left(  \eta_{l}\right)  \right)  \right]  .
\end{align*}
Due to LLN, the product in the square brackets goes to $\mu^{1}\left(
f_{1}\right)  \mu^{1}\left(  f_{2}\right)  \mu^{2}\left(  g_{1}\right)
\mu^{2}\left(  g_{2}\right)  $ in probability, so the theorem follows.
\end{proof}

\section{Conclusions}

The stochastic dominance technique introduced originally by A. Stolyar
\cite{St1} for the deterministic service time was extended here to the case of
a general service time distribution. Similar methods can hopefully be used for
the analysis of other mean-field models.

\end{document}